\documentclass[11pt,a4paper,english,reqno]{amsart}
\usepackage{amsmath,amssymb,amsfonts,epsfig,mathrsfs}
\usepackage[T1]{fontenc}

\usepackage{color}
\usepackage{array}
\usepackage{amsthm}
\usepackage{amstext}
\usepackage{graphicx}
\usepackage{setspace}
\usepackage[margin=2.5cm]{geometry}
\usepackage{bbm}
\usepackage{color}
\usepackage{enumitem}
\allowdisplaybreaks[4]

\usepackage{amscd,psfrag}
\usepackage{yhmath}
\usepackage[mathscr]{eucal}

\usepackage{slashed}

\makeatletter
\pdfpageheight\paperheight
\pdfpagewidth\paperwidth

\usepackage{epstopdf}
\usepackage{chngcntr}
\usepackage{mathrsfs}

\usepackage{indentfirst}	

\usepackage[normalem]{ulem}
\theoremstyle{plain}

\newtheorem{definition}{Definition}[section]
\newtheorem{theorem}[definition]{Theorem}
\newtheorem*{theorem*}{Theorem}

\newtheorem{remark}[definition]{Remark}

\newtheorem*{remark*}{Remark}
\newtheorem*{sideremark*}{Side Remark}\newtheorem*{mt*}{Main Theorem}

\newtheorem*{claim*}{Claim}
\newtheorem*{q*}{Question}
\newtheorem{lemma}[definition]{Lemma}

\newtheorem*{corollary*}{Corollary}
\newtheorem*{proposition*}{Proposition}

\newcommand{\R}{\mathbb{R}}

\newcommand{\na}{\nabla}
\newcommand{\dd}{{\rm d}}
\newcommand{\p}{\partial}

\newcommand{\map}{\rightarrow}

\newcommand{\1}{\mathbbm{1}}

\newcommand{\E}{{\mathbf{E}}}
\newcommand{\B}{{\mathbf{B}}}
\newcommand{\HH}{{\mathbf{H}}}
\newcommand{\pin}{{\psi_{\rm in}}}
\newcommand{\pout}{{\psi_{\rm out}}}
\newcommand{\tp}{{\widetilde{\phi}}}

\allowdisplaybreaks[4]

\def\XXint#1#2#3{{\setbox0=\hbox{$#1{#2#3}{\int}$ }
\vcenter{\hbox{$#2#3$ }}\kern-.6\wd0}}

\numberwithin{equation}{section}
\numberwithin{figure}{section}

\title{Volume decay and concentration of high-dimensional Euclidean balls  --- a PDE and variational perspective}

\author{Siran Li}
\address{Siran Li: Department of Mathematics, Rice University, MS 136
P.O. Box 1892, Houston, Texas, 77251, USA.}

\email{\texttt{Siran.Li@rice.edu}}



\date{\today}

\begin{document}

\maketitle

\begin{abstract}
It is a well-known fact --- which can be shown by elementary calculus --- that the volume of the unit ball in $\R^n$ decays to zero and simultaneously gets concentrated on the thin shell near the boundary sphere as $n \nearrow \infty$. Many rigorous proofs and heuristic arguments are provided for this fact from different viewpoints, including Euclidean geometry, convex geometry, Banach space theory, combinatorics,  probability, discrete geometry, etc. In this note we give yet another two proofs via the regularity theory of elliptic partial differential equations and calculus of variations. 

\end{abstract}

\bigskip
\section{The problem}

A well-known fact in high-dimensional Euclidean geometry, with which we may be familiar since the very first calculus class, can be stated as follows:
\begin{theorem}\label{thm}
Let $\B^n=\{x\in\R^n:|x|<1\}$ be the Euclidean unit ball in $\R^n$. The volume of $\B^n$ gets concentrated near the boundary sphere $\p \B^n=\{x\in\R^n:|x|=1\}$ and tends to $0$ as $n \nearrow \infty$.
\end{theorem}

Under the MathOverflow question ``\emph{What's a nice argument that shows the volume of the unit ball in $\R^n$ approaches $0$?}'' posted about $10$ years ago (\cite{x}), nearly a dozen elegant and surprising answers are provided. Contributors to the solutions and discussions include many renowned mathematicians: Greg Kuperberg, Timothy Gowers, Ian Agol, Bill Johnson, Gil Kalai, Pete L.~Clark, Anton Petrunin... The answers employ techniques from radically different fields of mathematics, ranging from combinatorics to the geometry of Banach spaces.

The aim of this note is give yet another two proofs of Theorem~\ref{thm} using the knowledge about harmonic functions and/or harmonic maps. More generally, we show that
\begin{theorem}\label{thm'}
For each $n=1,2,3,\ldots$ let $u_n$ be a harmonic function in $\B^n$. Assume that the Dirichlet energies of $u_n$ in $\B^n$ are uniformly bounded. Then the energies decay to $0$ and increasingly concentrate on $\p\B^n$ as $n \nearrow \infty$.
\end{theorem}

\begin{remark}\label{rem}
In view of Theorem~\ref{thm'}, our proof of Theorem~\ref{thm} shall assume \emph{a priori} that ${\rm Vol}(\B^n)$ are uniformly bounded in $n$. In fact, ${\rm Vol}(\B^n)$ are known to be maximised at $n=5$. On the other hand, we shall prove Theorem~\ref{thm'} for weakly harmonic functions, {\it i.e.} functions that satisfy the Laplace equation in the distributional sense.
\end{remark}


\section{The strategy}
Consider a harmonic function $u=(u^1,u^2,\ldots,u^n):\B^n \map \R^n$,
\begin{equation}\label{harmonic}
\Delta u^i = \sum_{j=1}^n \frac{\p^2 u^i}{\p x_j^2}=0
\end{equation}
for each $i\in\{1,2,\ldots, n\}$. A fundamental property of harmonic function is the \emph{energy decay} phenomenon: For each $R \in ]0,1[$, the Dirichlet energy $$\E(R):=\int_{\B_R} |\na u(x)|^2\,\dd x$$ satisfies 
\begin{equation}\label{decay}
\E(R/2)\leq \theta\E(R)
\end{equation}
for some number $\theta$ strictly less than $1$. Throughout $\B^n_R\equiv\B_R:=\{x\in\R^n:|x|<R\}$; we drop the superscript $n$ when there is no confusion about the dimension. By a standard argument in the regularity theory of elliptic PDEs, we may strengthen Eq.~\eqref{decay} to the following form: for any  $0<r<R\leq 1$ there holds
\begin{equation}\label{iterate}
\E(r)\leq C\Big(\frac{r}{R}\Big)^\beta \E(R),
\end{equation}
where $C$ is a universal constant (namely, independent of any parameters) and $\beta \simeq n$. As $n$ gets large, the factor $(r/R)^\beta$ decays exponentially. It means that, given two arbitrary concentric balls $\B_R$ and $\B_r$, the Dirichlet energy $\E_R$ is always concentrated in the shell $\B_R\sim\B_r$. One can now conclude by taking the identity harmonic map $u(x)=x$.

We may also deduce Theorem~\ref{thm} from  calculus of variations. It is well-known that harmonic functions (between Euclidean domains) are \emph{Dirichlet energy minimisers}: 
\begin{equation}\label{min}
{\bf Id}_{\B^n} = {\rm arg\,min} \Bigg\{\E[v]=\int_{\B^n}|\na v|^2\,\dd x:\, v \in W^{1,2}(\B^n)\text{ and } v(\omega)=\omega \text{ for all } \omega \in \p\B^n  \Bigg\}. 
\end{equation}
Here $W^{1,2}(\B^n)$ denotes the Sobolev space of finite-energy maps:
\begin{equation}
W^{1,2}(\B^n) := \Bigg\{w:\B^n\to\R^n:\, \int_{\B^n} \Big(|\na w|^2 + |w|^2\Big)\,\dd x \Bigg\}.
\end{equation}

Eq.~\eqref{min} is tantamount to the stationariness of  $u={\bf Id}_{\B^n}$ with respect to both \emph{inner} and \emph{outer variations}, {\it i.e.}, the one-parameter families of smooth variations which deform and the domain and the range of $u$, respectively. These together imply that
\begin{equation}\label{1/n-2}
\E(1) \leq c_1 \int_{\p B^n} |\na_{\rm tan} u|^2\,\dd \Sigma,
\end{equation}
where $\na_{\rm tan}$ and $\dd\Sigma$ are respectively the  gradient and surface measure on the unit sphere; $c_1 \sim \mathcal{O}(1/n)$. From here, a rescaling and iteration argument as before will lead us to Eq.~\eqref{iterate}.

In the following two sections we make the above discussions rigorous, thus giving two more proofs of Theorem~\ref{thm}. Our arguments can be found, in one form or another, in any standard textbook on elliptic PDEs and calculus of variations. We refer the readers to \cite{hl} by Qing Han and Fang-Hua Lin, and \cite{s} by Leon Simon, among many other references. For background materials on mollifiers and elementary inequalities, see \cite{ll} by Elliott Lieb and Michael Loss.

\section{The PDE proof}


\subsection{Gradient estimate} Let us take $u\in W^{1,2}(\B^n)$ to be any weak ({\it i.e.}, distributional) solution for the Laplace equation~\eqref{harmonic}.
We show for all large $p$ that the $L^p$-norm of $\na u$ over $\B_{1/2}$ can be controlled by the $L^2$-norm of $u$ over $\B_1=\B^n$.

\begin{lemma}\label{lem: weyl estimate}
For each $p \in ]2,\infty[$, there is a  constant $C_2$ depending only on $p$ such that
$$\|\na u\|_{L^p(\B_{1/2})} \leq C_2\|u\|_{L^2(\B_1)}.$$
\end{lemma}
\begin{proof}
It is well-known that 
harmonic functions satisfy the mean-value property. So, for a symmetric mollifier $J$ on $\R^n$, pointwise we have $u=J_\delta \star u$ for each $\delta \in ]0,1/2]$, where $J_\delta(x):=\delta^{-n}J(x/\delta)$ and $\star$ is the convolution. Thus, by Young's convolution inequality  we can bound
\begin{align*}
\|\na u\|_{L^p(\B_{1/2})} &\leq \|\na J_\delta\|_{L^{q}(\B_{1/2})} \|u\|_{L^2(\B_1)},
\end{align*}
where $q$ is determined by $1/q = 1/p+1/2$. A simple scaling argument gives us $\|\na J_\delta\|_{L^{q}(\B_{1/2})} \leq \delta^{-1} \|\na J\|_{L^{q}(\B_1)}$. Now one may complete the proof by fixing $\delta$ and $J$.  \end{proof}

\subsection{Energy decay}
Next let us deduce that
\begin{lemma}\label{lem: energy decay}
For any $0 \leq r <R \leq 1$ there holds
\begin{equation} \label{energy decay estimate}
\E(r) \leq C_3 \Big(\frac{r}{R}\Big)^\beta \E(R).
\end{equation}
Here $C_3$ is a universal constant, and $\beta \in ]n/2,n[$ is a dimensional constant. In fact, $\beta$ can be chosen as close to $n$ as we want.
\end{lemma}

\begin{proof}
By considering $u_R(x):=u(x/R)$ it suffices to prove for $R=1$ and $r \in ]0,1/2]$. We apply the H\"{o}lder inequality,  Lemma~\ref{lem: weyl estimate}, and the scaling ${\rm Vol}(\B_r)\slash{\rm Vol}(\B_1^n)=r^n$ to obtain
\begin{align*}
\E(r) &:= \int_{\B_{r}} |\na u|^2\,\dd x \\
&\leq \bigg\{\int_{\B_{1/2}} |\na u|^{2p}\,\dd x\bigg\}^{\frac{1}{p}} \big[{\rm Vol}(\B_r)\big]^{\frac{p-1}{p}}\\
&\leq (C_2)^2 \|u\|_{L^2(\B_1)}^2\big[{\rm Vol}(\B_r)\big]^{\frac{p-1}{p}}\\
&\equiv (C_2)^2 r^{\frac{n(p-1)}{p}}\big[{\rm Vol}(\B_1^n)\big]^{\frac{p-1}{p}}\E(1).
\end{align*}
Here $p$ is an arbitrary number in $]2,\infty[$. We select $\beta:=\frac{n(p-1)}{p}$ and note that $\beta \nearrow n$ as $p \nearrow \infty$. In addition, the volume of the unit ball is uniformly bounded in $n$; hence, there is a universal constant $C_3$ which bounds $(C_2)^2\big[{\rm Vol}(\B_1^n)\big]^{\frac{p-1}{p}}$ from the above. The proof is now complete.  \end{proof}

\subsection{Conclusion}
Now we are at the stage of presenting
\begin{proof}[Proof of Theorem~\ref{thm} and Theorem~\ref{thm'}]

 By Lemma~\ref{lem: energy decay}, for each $r \in [0,1[$ one has $\E(r) \leq C_3 r^\beta \E(1)$. Since $\beta \nearrow \infty$ and $C_3$ is universal, we have $\E(r) \searrow 0$  as $n \nearrow \infty$. Since $r \in [0,1[$ is arbitrary, we can conclude that $\E(1) \searrow 0$. Energy concentration follows directly from Eq.~\eqref{energy decay estimate}. Hence Theorem~\ref{thm'} is proved. On the other hand, clearly $u={\bf Id}_{\B^n}$ is a harmonic function. Its Dirichlet energy is given by
\begin{align*}
\E(r) = \int_{\B_r}|\na x|^2\,\dd x = n {\rm Vol}(\B_r).
\end{align*}
Sending $r \nearrow 1$, we find that ${\rm Vol}(\B^n)$ decays no slower than $ \mathcal{O}(1/n)$. This  yields Theorem~\ref{thm}.  \end{proof}

\section{The variational proof}

\subsection{Inner and outer variations}
It is well-known that a harmonic function $u:\B^n \to \R^n$ is a \emph{Dirichlet-minimiser}; that is, $u$ minimises $\E(1)$ among all the finite-energy maps attaining the same values on $\p\B^n$ (see Eq.~\eqref{min}). In particular, consider the following two types of variations: 
\begin{itemize}
\item
({\bf Inner variation}). Consider $\phi \in C^\infty_0(\B^n,\R^n)$ and $\pin^t(x) := x+t\phi(x)$. 
\item
({\bf Outer variation}). Consider $\tp \in C^\infty(\B^n \times \R^n;\R^n)$ such that $\tp(x,u)=0$ near $\p\B^n\times \R^n$, $|\na_u\tp(x,u)| \leq C_4$ and $|\tp(x,u)|+|\na_x\tp(x,u)|\leq C_5(1+|u|)$ for universal constants $C_4$ and $C_5$. Then we set $\pout^t(x,u):=u(x)+t\tp(x,u)$.
\end{itemize}

Here, $\pin^t$ and $\pout^t$ are one-parameter families of boundary-preserving diffeomorphisms obtained by deforming the domain and the range of $u$, respectively. The minimality of $u$ yields 
\begin{eqnarray}
&& \frac{d}{dt}\bigg|_{t=0} \Bigg\{\int_{\B^n} \Big| \na \big(u \circ \pin^t(x)\big) \Big|^2\,\dd x\Bigg\}=0,\label{in}\\
&& \frac{d}{dt}\bigg|_{t=0}\Bigg\{ \int_{\B^n} \Big| \na  \pout^t(x,u)\Big|^2\,\dd x\Bigg\}=0.\label{out}
\end{eqnarray}

As is standard in calculus of variation, we take $\phi(x)=\eta(|x|)x$ and $\tp(x,u)=\eta(|x|)u$. If, furthermore, the test function $\eta$ is chosen to tend to the indicator function $\1_{[0,r[}$, then a direct computation from Eqs.~\eqref{in} and \eqref{out} gives us
\begin{eqnarray}
&& (n-2)\int_{\B_r}|\na u|^2\,\dd x = r\int_{\p\B_r} |\na u|^2\,\dd \Sigma - 2r \int_{\p\B_r} |\p_\nu u|^2\,\dd\Sigma,\label{in'}\\
&& \int_{\B_r}|\na u|^2\,\dd x = \sum_{i=1}^n \int_{\p \B_r} u^i(\p_\nu u)^i \,\dd\Sigma. \label{out'}
\end{eqnarray}
In the above, $r \in [0,1]$ is arbitrary, $\nu$ is the outward unit normal vectorfield, and $\dd\Sigma$ is the (Riemannian) surface measure as before.

\subsection{Proof of Theorems~\ref{thm'} and \ref{thm}} In this subsection we show 
\begin{lemma}\label{lem: minimiser}
If $u:\B^n\to\R^n$ is a non-constant Dirichlet minimiser for $n\geq 3$, then 
\begin{equation}
\E(1) <\frac{2}{n-2} \HH(1).
\end{equation}
\end{lemma}
\begin{definition}
$\HH$ denotes the surface-Dirichlet energy:
\begin{equation*}
\HH(r):= \int_{\p\B_r} |\na_{\rm tan} u|^2\,\dd\Sigma.
\end{equation*}
$\na_{\rm tan}$ is the tangential gradient on $\p\B_r$, {\it i.e.}, the gradient associated to the Levi-Civita connection on the round sphere $\p\B_r$.
\end{definition}

Assuming Lemma~\ref{lem: minimiser}, we may immediately deduce Theorems~\ref{thm'} and \ref{thm}.
\begin{proof}[Proof of Theorem~\ref{thm} and Theorem~\ref{thm'}]
By assumption, $\HH(1)$ is bounded independent of $n$. Hence $\E(1) \searrow 0$ as $n \nearrow \infty$. The identity map $u(x)=x$ is a Dirichlet minimiser with $\E(r)=n {\rm Vol}(\B_r)$, so we have ${\rm Vol}(\B^n)\searrow 0$. In fact, it follows that ${\rm Vol}(\B^n)$ decays no slower than $\mathcal{O}(1/n^2)$.   \end{proof}

What's left now is to present a proof of Lemma~\ref{lem: minimiser}. It follows fairly straightforwardly from the formulae of inner and outer variations, Eqs.~\eqref{in'} and \eqref{out'}.

\begin{proof}[Proof of Lemma~\ref{lem: minimiser}]
Assume for contradiction that $\E(1) \geq \frac{2}{n-2}\HH(1)$. We may compute $\HH(1)$ by subtracting the angular derivatives from the total derivatives:
\begin{align*}
\HH(1) = \int_{\p\B^n}\Big( |\na u|^2 - |\p_\nu u|^2\Big)\,\dd \Sigma.
\end{align*}
By Eq.~\eqref{in'} for inner variations we get
\begin{align*}
\HH(1) &= \int_{\p\B^n}|\na u|^2 \,\dd \Sigma + \frac{n-2}{2}\int_{\B^n} |\na u|^2\,\dd x - \frac{1}{2}\int_{\p\B^n}|\na u|^2\,\dd \Sigma\\
&= \frac{1}{2}\int_{\p\B^n}|\na u|^2 \,\dd \Sigma + \frac{n-2}{2}\int_{\B^n} |\na u|^2\,\dd x,
\end{align*}
which is no greater than $(n-2)\E(1)/2$ by assumption. It implies that 
\begin{align*}
\int_{\p\B^n}|\na u|^2 \,\dd \Sigma=0.
\end{align*}
But this forces $\p_\nu u$ to vanish in the $L^2$-norm on $\p\B^n$, which in turn implies that $\E(1)=0$ by the outer variation Eq.~\eqref{out'}. Thus $u$ is a constant on $\B^n$.  \end{proof}


\end{document}